\documentclass[12 pt]{article}

\usepackage{CJK}
\usepackage{amssymb}
\usepackage{bbm}
\usepackage{amsfonts}
\usepackage{mathrsfs}
\usepackage{graphicx}
\usepackage{amsmath}
\usepackage{amsthm}
\usepackage{xypic}
\usepackage{times}
\usepackage{geometry}
\usepackage{color}
\usepackage{indentfirst}
\usepackage{cases}
\usepackage{hyperref}
\usepackage{float}
\usepackage{authblk}

\hypersetup{
    colorlinks=true,
    linkcolor=blue,
    citecolor=blue
}
\newtheorem{Theorem}{\scshape  Theorem}[section]
\newtheorem{Lemma}[Theorem]{\scshape  Lemma}

\title{Graphs with maximal Laplacian eigenvalue multiplicity}

\date{}

\makeatletter
\renewenvironment{abstract}{%
    \if@twocolumn
        \section*{\abstractname}%
    \else
        \begin{center}%
            {\bfseries \abstractname\vspace{0.5mm}}%
        \end{center}%
        \quotation
    \fi}
    {\if@twocolumn\else\endquotation\fi}
\makeatother

\begin{document}
\baselineskip 17pt

\title{Graphs with maximal Laplacian eigenvalue multiplicity}

\author{
Songnian Xu$^{1}$ \thanks{Corresponding author. E-mail address: xsn131819@163.com},
Xiaoya Li$^{1}$,
Dein Wong$^{1}$ \thanks{Corresponding author. E-mail address: wongdein@163.com. Supported by the National Natural Science Foundation of China (No.~12371025)},
Wenhao Zhen$^{2}$
}

\affil{$^{1}$ Department of Mathematics, China University of Mining and Technology, Xuzhou 221116, P.R. China}
\affil{$^{2}$ School of Mathematics and Statistics, Zaozhuang University, Zaozhuang 277160, P.R. China}
\date{}
\maketitle

\vspace{-1cm}

\begin{abstract}
In this paper, \( G \) is a simple connected graph, and \( m_G(\lambda) \) denotes the multiplicity of \( \lambda \) as an eigenvalue of the Laplacian matrix \( L(G) \). Let \( p(G) \) denote the number of pendant vertices of \( G \), \( q(G) \) the number of quasi-pendant vertices of \( G \), and \( c(G) \) the dimension of the cycle space of \( G \). 

Li et al. [Discrete Mathematics, 2026] proved that if \( G \) is a tree with \( G \not\cong K_{1,n-1} \) and \( \lambda \neq 1 \), then 
\[
m_G(\lambda) \le q(T) - 1.
\]
Moreover, for a general graph \( G \) with \( \lambda \neq 1 \), Li et al. also proved in the same paper that 
\[
m_G(\lambda) \le c(G) + q(G),
\]
with equality if and only if \( G \cong K_{1,n-1} \) or \( G \cong C_n \) with \( \lambda \notin \{0, 4\} \).

A natural consequence is that if \( c(G) + q(G) \ge 2 \), then 
\[
m_G(\lambda) \le 2c(G) + q(G) - 1.
\]
When \( c(G) = 0 \), this reduces to the result of Li et al. for trees. In this paper, we give a complete characterization of graphs \( G \) attaining the equality 
\[
m_G(\lambda) = 2c(G) + q(G) - 1.
\]
\end{abstract}

\vspace{-3mm}

\let\thefootnoteorig\thefootnote
\renewcommand{\thefootnote}{\empty}

\noindent{\bf AMS classification:} 05C50
\vskip 1.5mm
\noindent{\bf Keywords:} Laplacian matrix; eigenvalue multiplicity; cyclomatic number; quasi-pendant vertices

\vspace{-2mm}

\section{Introduction}
In this paper, we only consider simple, connected finite graphs $G$.
Let $V(G)$ and $E(G)$ denote the sets of vertices and edges of the graph $G$, respectively.
The order of $G$, which refers to the number of vertices in $G$, is denoted as $|G|$.
A graph $H$ is called an induced subgraph of $G$ if any two vertices are adjacent in $H$ if and only if they are adjacent in $G$.
For $K\subseteq V(G)$, we denote the subgraph induced by $K$ as $[K]$.
We often use $G-K$ to represent the graph $[V(G) \setminus V(K)]$.
The degree of a vertex $v$ in $G$ is denoted by $d_G(v)$ , which represents the number of adjacent vertices to $v$ in $G$.
If $d_G(v) \geq 3$, we refer to $v$ as a principal vertex of $G$.
A vertex \( v \) is called a pendant vertex if \( d_G(v) = 1 \). A vertex adjacent to a pendant vertex is called a quasi-pendant vertex. Let \( p(G) \) and \( q(G) \) denote the number of pendant vertices and quasi-pendant vertices of \( G \), respectively. The cyclomatic number of \( G \) is defined as \( c(G) = |E(G)| - |V(G)| + m \), where \( m \) is the number of connected components of \( G \). Since we only consider simple connected graphs in this paper, we have \( m = 1 \).

Let \( A(G) = (a_{ij})_{n \times n} \) denote the adjacency matrix of a graph \( G \), where \( a_{ij} = 1 \) if \( v_i \sim v_j \), and \( a_{ij} = 0 \) otherwise. The Laplacian matrix of \( G \) is defined as \( L(G) = D(G) - A(G) \), where \( D(G) \) is the diagonal degree matrix with \( d_{ii} = |N_G(v_i)| \) and \( d_{ij} = 0 \) for \( i \neq j \).
Let \( m_{L(G)}(\lambda) \) denote the multiplicity of \( \lambda \) as an eigenvalue of the Laplacian matrix \( L(G) \), which we abbreviate as \( m_G(\lambda) \) when no confusion arises.

The multiplicity of Laplacian eigenvalues has attracted considerable attention in recent decades. In 1985, Faria \cite{Far} first considered the multiplicity of the Laplacian eigenvalue \( 1 \), and established the inequality
\[
m_G(1) \ge p(G) - q(G).
\]
 Subsequently, Grone et al. \cite{Gro} showed that for any tree \( T \),
\[
m_T(1) \le p(T) - 1.
\]

A number of subsequent works have been devoted to characterizing graphs that attain certain bounds on Laplacian eigenvalue multiplicities. Andrade et al. \cite{Andr} provided a partial characterization of graphs \( G \) satisfying \( m_G(1) = p(G) - q(G) \). Gupta \cite{Gup} proved that \( m_T(1) = p(T) - 1 \) for a tree \( T \) if and only if \( d(u_1, u_2) \equiv 2 \pmod{3} \) for any two distinct pendant vertices \( u_1 \) and \( u_2 \), and later in \cite{Gup1} gave a complete characterization of trees satisfying \( m_T(1) = p(T) - 2 \). In the same year, Wong et al. \cite{Wong} characterized the trees \( T \) and the Laplacian eigenvalues \( \lambda \) for which \( m_T(\lambda) = p(T) - 1 \).
Guo et al. \cite{Guo} proved that if \( T \) is a tree with \( n \) vertices, then 
\( m_T(1) \in \{0, 1, \ldots, n - 5, n - 4, n - 2\} \), and moreover, for each 
\( k \) in this set, there exists a tree \( T \) with \( n \) vertices such that \( m_T(1) = k \).
In 2025, Wang et al. \cite{Wang} proved that if \( T \) is a tree with \( n \) vertices and \( \beta'(T) \ge 2 \), then
\[
m_T(1) \le n - 2\beta'(T) - 1,
\]
where \( \beta'(T) \) is the induced matching number of \( T \), and they characterized the extremal trees attaining equality. 
Tian et al. \cite{Tian1} proved that
\[
m_T(1) = p(T) - q(T) + m_{\overline{T}}(1),
\]
where \( \overline{T} \) is the reduced tree of \( T \), obtained from \( T \) by deleting some pendant vertices such that \( p(\overline{T}) = q(\overline{T}) \). Furthermore, for every reduced tree \( T \) with \( n \ge 6 \) vertices, they proved that
\[
m_T(1) \le \frac{n-2}{4},
\]
and completely characterized the extremal trees attaining this bound.
More general upper bounds involving the cyclomatic number have also been investigated. In 2022, Wen et al. \cite{Wen} proved that if \( G \) is not a cycle, then
\[
m_G(\lambda) \le 2c(G) + p(G) - 1.
\]
 Han et al. \cite{Han} further characterized the graphs \( G \) with \( m_G(1) = 2c(G) + p(G) - 1 \).
Very recently, Li et al. \cite{Li2026} investigated upper bounds for the multiplicity of Laplacian eigenvalues of graphs. They proved that for any tree \( T \not\cong K_{1,n-1} \) with \( \lambda \neq 1 \),
\[
m_T(\lambda) \le q(T) - 1,
\]
and completely characterized the trees attaining equality. Moreover, for any connected graph \( G \), they established the general bound
\[
m_G(\lambda) \le 2c(G) + q(G) \quad (\lambda \neq 1),
\]
with equality if and only if \( G \) is either a star \( K_{1,n-1} \) or a cycle \( C_n \) with \( \lambda \notin \{0, 4\} \). 
According to the result of Li et al., we know that if \( c(G) + q(G) \ge 2 \) and \( \lambda \neq 1 \), then necessarily
\[
m_G(\lambda) \le 2c(G) + q(G) - 1.
\]
In this paper, we give a complete characterization of the extremal graphs attaining equality.

\section{Main resulas}

Let \( m \) be a positive integer. It is well known that:

(1) The Laplacian eigenvalues of \( C_g \) are
\[
2 - 2\cos\frac{2k\pi}{g}, \qquad k = 1, 2, \ldots, g.
\]

(2) The Laplacian eigenvalues of \( P_l \) are
\[
2 - 2\cos\frac{k\pi}{l}, \qquad k = 0, 1, \ldots, l-1.
\]

This implies that if there exist two coprime integers \( i \) and \( m \) with \( 1 \le i \le m \) such that
\[
\lambda = 2 - 2\cos\frac{2i\pi}{m},
\]
then:
\[
\lambda \text{ is an eigenvalue of } C_g \iff g \equiv 0 \pmod m,
\]
and
\[
\lambda \text{ is an eigenvalue of } P_l \iff 2l \equiv 0 \pmod m.
\]

$T(k,l)$: When $l\geq1$, let $T(k,l)$ be the tree on $k + l$ vertices obtained from $K_{1,k}$ by identifying its center with a
pendant vertex of $P_l$.
We call another pendant vertex of $P_l$ the $l$-vertex of $T(k,l)$.
In particular, $T(k,1) = K_{1,k}$ and $T(k,0)$ is the disjoint union of $k$ copies of $P_1$.

$C(k,l)$: When $l\geq1$, let $C(k,l)$ be the graph on $k + l-1$ vertices obtained from $C_k$ by identifying its a vertex with a
pendant vertex of $P_l$.
We call another pendant vertex of $P_l$ the $l$-vertex of $C(k,l)$.
In particular, $C(k,1) =C_k$ and $C(k,0)=P_{k-1}$.

$pendant$-$T(k,l)$ and $pendant$-$C(g,l)$:
Let $H$ be a graph, $l\geq1$ and $u \in V(H)$ such that $d_H(u) \geq 2$.
If $v$ is an $l$-vertex of $T_{k,l}$ and $G$ is the graph obtained by adding an edge between $u$ and $v$, then we refer to $T(k,l)$ as a $pendant$-$T(k,l)$ of $G$.
Similarly, if $v$ is an $l$-vertex of $C(g,l)$ and $G$ is the graph obtained by adding an edge between $u$ and $v$, then we refer to $C(g,l)$ as a $pendant$-$C(g,l)$ of $G$.

A graph $G$ is called $\lambda$-$optimal$ if $m_G(\lambda) = 2c(G) + q(G) - 1$.
Since \( c(G) = 0 \) means that \( G \) is a tree, and the tree case has been completely settled in \cite{Li2026}, in this paper we only consider the case \( c(G) \ge 1 \). 
During the proof, we observe that the characterization of \( \lambda \)-optimal graphs differs between the cases \( c(G) + q(G) = 2 \) and \( c(G) + q(G) \ge 3 \). This distinction can be seen from Theorems 2.1 and 2.2 below.

 For convenience, let \( F_{uv}(G, H; P_l) \) denote the graph obtained from two graphs \( G \) and \( H \), together with a path \( P_l : w_1 w_2 \cdots w_l \) (\( l \geq 1 \)), by identifying the vertex \( u \) of \( G \) with \( w_1 \) and the vertex \( v \) of \( H \) with \( w_l \), respectively. In the case \( l = 1 \), this graph is precisely the coalescence of \( G \) and \( H \) formed by identifying the vertex \( u \) of \( G \) with the vertex \( v \) of \( H \).

\begin{Theorem}
Let \( c(G) + q(G) = 2 \) with \( c(G) \ge 1 \) and \( \lambda \neq 1 \). Then \( G \) is \( \lambda \)-optimal if and only if one of the following holds:

(1) \( G = F_{u,w}(C, K_{1,k}; P_l) \), where \( l > 1 \), \( m_C(\lambda) = 2 \), and \( m_{T_{k,l-1}}(\lambda) = 1 \);

(2) \( G = F_{uw}(C_{g_1}, C_{g_2}; P_l) \), where \( l > 1 \), \( m_{C_{g_1}}(\lambda) = m_{C_{g_2}}(\lambda) = 2 \), and when \( l > 2 \), \( m_{P_{l-2}}(\lambda) = 1 \).
\end{Theorem}

Denote by \( L_v(G) \) the submatrix of \( L(G) \) obtained by removing the row and column corresponding to the vertex \( v \). Let \( B_{k,l} \) be the matrix of order \( k+l \) obtained from \( L(T(k,l+1)) \) by deleting the row and column corresponding to the \((l+1)\)-vertex of \( T(k,l+1) \). 
Let \( Q_{g,l} \) be the square matrix of size \( g+l \) obtained from \( L(C(g,l+1)) \) by deleting the row and column corresponding to the \((l+1)\)-vertex of \( C(k,l+1) \). 
Let \( U_n \) be the matrix of order \( n \) obtained from \( L(P_{n+2}) \) by deleting the rows and columns corresponding to both end vertices of \( P_{n+2} \).

In the proof of Section 3, we will find that if \( c(G) + q(G) \ge 3 \), \( \lambda \neq 1 \), and \( G \) is \( \lambda \)-optimal, then every \( C(k,l) \) of \( G \) is a pendant-\( C(k,l) \), and every \( T(k,l) \) is a pendant-\( T(k,l) \) (see Lemma 3.9 for the precise result). To better present our conclusions, we assume henceforth that the conclusion of Lemma 3.9 is known.
For a graph \( G \), define \( \Gamma(G) \) as the set of vertices of degree at least 3 in \( G \), except principal vertices on cycles and quasi-pendant vertices.
In the following discussion, for a graph \( G \) with \( c(G) + q(G) \ge 3 \), we let
\[
G - \Gamma(G) = \mathcal{C}_G \cup \mathcal{T}_G \cup \mathcal{P}_G,
\]
where \( \mathcal{C}_G \) is the union of \( C(g_m, h_m) \) for \( 1 \le m \le c(G) \), \( \mathcal{T}_G \) is the union of \( T(k_i, l_i) \) for \( 1 \le i \le q(G) \), and \( \mathcal{P}_G \) is the union of \( P_{t_j} \) for \( 1 \le j \le |\Gamma(G)| - 1 \).

\begin{Theorem}
    
Let \( G \) be a graph with \( c(G) \ge 1 \), \( \lambda \neq 1 \), and \( c(G) + q(G) \ge 3 \). Suppose
\[
G - \Gamma(G) = \mathcal{C}_G \cup \mathcal{T}_G \cup \mathcal{P}_G.
\]
Then \( G \) is \( \lambda \)-optimal if and only if
\[
\begin{aligned}
&m_{Q_{g_m,h_m}}(\lambda) = 2 &&(1 \le m \le c(G)), \\
&m_{B_{k_i,l_i}}(\lambda) = 1 &&(1 \le i \le q(G)), \\
&m_{U_{t_j}}(\lambda) = 1 &&(1 \le j \le |\Gamma(G)| - 1).
\end{aligned}
\]
\end{Theorem}

The following theorem gives equivalent conditions for
\[
m_{Q_{g_m,h_m}}(\lambda) = 2 \quad (1 \le m \le c(G)).
\]

\begin{Theorem}
Let \( G \) be a graph with \( c(G) \ge 1 \), \( \lambda \neq 1 \), and \( c(G) + q(G) \ge 3 \). Suppose
\[
G - \Gamma(G) = \mathcal{C}_G \cup \mathcal{T}_G \cup \mathcal{P}_G.
\]
Then \( G \) is \( \lambda \)-optimal if and only if
\[
\begin{aligned}
&m_{C_{g_m}}(\lambda) = 2,\ l_m > 2,\ \text{and } m_{L_v(P_{l_m})}(\lambda) = 1 &&(1 \le m \le c(G)), \\
&\quad \text{where } v \text{ is an arbitrary pendant vertex of } P_{l_m}, \\
&m_{B_{k_i,l_i}}(\lambda) = 1 &&(1 \le i \le q(G)), \\
&m_{U_{t_j}}(\lambda) = 1 &&(1 \le j \le |\Gamma(G)| - 1).
\end{aligned}
\]
\end{Theorem}

\section{Preliminaries}

\begin{Lemma}[{\cite[Theorem 3.2]{Mo}}]
\label{lem:3.1}
Let \( G \) be a graph with \( n \) vertices and let \( e \) be an edge of \( G \). Then
\[
\lambda_1(G) \ge \lambda_1(G-e) \ge \lambda_2(G) \ge \lambda_2(G-e) \ge \cdots \ge \lambda_n(G) \ge \lambda_n(G-e).
\]
\end{Lemma}

 \begin{Lemma}
\label{lem:3.2}
Assume \( G = F_{uv}(H_1, H_2; P_l) \). Then the following bounds hold:

(1)\cite[Lemma 2.6]{Li2026} If \( H_1 = K_{1,k} \), then \( m_G(\lambda) \le m_{H_2}(\lambda) + 1 \).

(2)\cite[Lemma 2.8]{Li2026} If \( H_1 = C_g \), then \( m_G(\lambda) \le  m_{H_2}(\lambda) + 2 \).
\end{Lemma}

\begin{Lemma}
\label{lem:3.3}\cite[Lemma 2.11]{Li2026}
For any \( \lambda \in \sigma(B_{k,l}) \) with \( \lambda \neq 1 \), we have \( \lambda \notin \sigma(B_{k,l-1}) \).
\end{Lemma}

For $\lambda\in \mathbb{R}$, let
 $$\mathbb{V}^{\lambda}_{L(G)}=\{\alpha\in \mathbb{R}_n| L(G)\alpha=\lambda \alpha\}$$
 $$\mathbb{Z}^{X}_{L(G)}= \{\alpha\in \mathbb{R}_n| \alpha_x = 0, x\in X\}.$$

For a graph \( G \), if \( \lambda \) is a Laplacian eigenvalue of \( G \) and \( X \) is an eigenvector corresponding to \( \lambda \), then the characteristic equation is \( L(G)X = \lambda X \). This means, for \( v \in V(G) \),
\begin{equation}
(d_v - \lambda)x_v = \sum_{w \in N_G(v)} x_w, \tag{1}
\end{equation}
where \( x_v \) denotes the component of \( X \) at vertex \( v \).

\begin{figure}[H]
  \centering
  \includegraphics[width=0.6\linewidth]{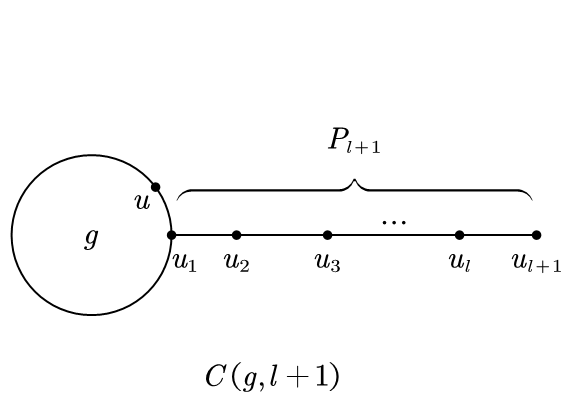}
\caption{$ G = C(g,l+1)$}\label{figure1}
\end{figure}

\begin{Lemma}
\label{lem:3.4}
If \( \lambda \neq 1 \), $l\ge 1$ and \( m_{Q_{g,l}}(\lambda) = 2 \), then \( m_{Q_{g,l-1}}(\lambda) = 1 \).
\end{Lemma}
\begin{proof}
Let $\Phi(G; \lambda)=(\lambda I-L(G))$.
Note that
\begin{equation}
\Phi(C(g,l); \lambda) =(\lambda -1) \Phi(Q_{g,l-1}; \lambda) - \Phi(Q_{g,l-2}; \lambda), \tag{2}
\end{equation}
and
\begin{equation}
\Phi(Q_{g,l}; \lambda) = \Phi(C(g,l); \lambda) - \Phi(Q_{g,l-1}; \lambda). \tag{3}
\end{equation}

First, by the interlacing theorem, we know that \( m_{Q_{g,l-1}}(\lambda) \ge 1 \). 
Next we prove that \( m_{Q_{g,l}}(\lambda) \le 2 \). Let \( C(g,l+1) \) be as shown in Figure 1, and suppose \( |V(C(g,l+1))| = n+1 \). Let 
\[
X \in \mathbb{V}^{\lambda}_{Q_{g,l}} \cap \mathbb{Z}^{\{u_l, u\}}_{Q_{g,l}},
\]
and let \( x_{v_i} \) denote the component of \( X \) at vertex \( v_i \). Since \( x_{u_l} = 0 \), by equation (1), we obtain \( x_{u_{l-1}} = 0 \). Continuing this process, we get \( x_{u_i} = 0 \) for \( 1 \le i \le l \). Together with \( x_u = x_{u_1} = 0 \), repeatedly applying equation (1) yields \( X = 0 \), which implies that
\[
\dim\left(\mathbb{V}^{\lambda}_{Q_{g,l}} \cap \mathbb{Z}^{\{u_l, u\}}_{Q_{g,l}}\right) = 0.
\]

Using the dimension formula,
\[
\dim\left(\mathbb{V}^{\lambda}_{Q_{g,l}} \cap \mathbb{Z}^{\{u_l, u\}}_{Q_{g,l}}\right)
= \dim(\mathbb{V}^{\lambda}_{Q_{g,l}}) + \dim(\mathbb{Z}^{\{u_l, u\}}_{Q_{g,l}}) - \dim\left(\mathbb{V}^{\lambda}_{Q_{g,l}} \cup \mathbb{Z}^{\{u_l, u\}}_{Q_{g,l}}\right).
\]
Note that
\[
\dim(\mathbb{Z}^{\{u_l, u\}}_{Q_{g,l}}) = n-2,
\qquad
\dim\left(\mathbb{V}^{\lambda}_{Q_{g,l}} \cup \mathbb{Z}^{\{u_l, u\}}_{Q_{g,l}}\right) \le n.
\]
Therefore,
\[
\dim\left(\mathbb{V}^{\lambda}_{Q_{g,l}} \cap \mathbb{Z}^{\{u_l, u\}}_{Q_{g,l}}\right)
\ge \dim(\mathbb{V}^{\lambda}_{Q_{g,l}}) + n - 2 - n
= \dim(\mathbb{V}^{\lambda}_{Q_{g,l}}) - 2.
\]
That is,
\[
\dim(\mathbb{V}^{\lambda}_{Q_{g,l}}) \le \dim\left(\mathbb{V}^{\lambda}_{Q_{g,l}} \cap \mathbb{Z}^{\{u_l, u\}}_{Q_{g,l}}\right) + 2 = 2.
\]
We now assume for contradiction that \( m_{Q_{g,l-1}}(\lambda) = 2 \).
Since \( m_{Q_{g,l}}(\lambda) = 2 \), it follows from (3) that \( \lambda \) is a double root of \( \Phi(C(g,l); \lambda) \). Then by (2), we obtain \( m_{Q_{g,l-2}}(\lambda) = 2 \).

By repeating the above process iteratively, we eventually obtain that \( \lambda \) is a double root of \( \Phi(L_{u_1}(C_g); \lambda) \). However, we know that
\[
L_{u_1}(C_g) = 2 - A(P_{g-1}),
\]
which implies that \( L_{u_1}(C_g) \) has no double eigenvalue, a contradiction.
Therefore, we must have \( m_{Q_{g,l-1}}(\lambda) = 1 \).
\end{proof}

\begin{Lemma}\label{lem:PW}
\cite{PW}
Let \( v \) be a vertex of a graph \( H \) and \( u \) be a vertex of \( G \).
 If 
\[
m_{L_v(F_{uv}(G, P_1; P_2))}(\lambda) = m_{L_u(G)}(\lambda) + 1,
\]
then
\[
m_{F_{uv}(G,H; P_2)}(\lambda) = m_{L_u(G)}(\lambda) + m_{L_v(H)}(\lambda).
\]
\end{Lemma}

\begin{figure}[H]
  \centering
  \includegraphics[width=0.6\linewidth]{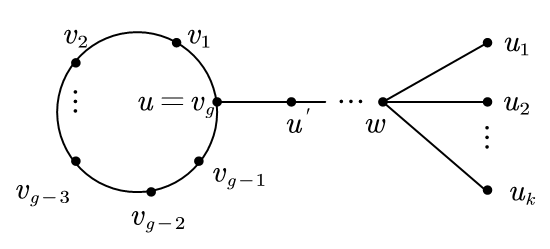}
\caption{$ G = F_{u,w}(C, K_{1,k}; P_l)$}\label{figure2}
\end{figure}

\begin{Lemma}
Let \( G = F_{u,w}(C, K_{1,k}; P_l) \) as shown in Figure 2 and $\lambda\neq 1$. Then \( G \) is \( \lambda \)-optimal if and only if \( l > 1 \), \( m_C(\lambda) = 2 \), and \( m_{T_{k,l-1}}(\lambda) = 1 \).
\end{Lemma}

\begin{proof}
The vertices of \( G \) are labeled as in Figure 2. 
First, we assume that \( G \) is \( \lambda \)-optimal.
Let \( X \in \mathbb{V}_{L(G)}^\lambda \cap \mathbb{Z}_{L(G)}^w \), and let \( x_{v_i} \) denote the component of \( X \) at vertex \( v_i \). 
Since \( m_G(\lambda) = 2 \), by the dimension formula we have
\[
\dim(\mathbb{V}_{L(G)}^\lambda \cap \mathbb{Z}_{L(G)}^w) \ge \dim(\mathbb{V}_{L(G)}^\lambda) - 1 = 1,
\]
so there exists such a nonzero vector \( X \).
Using equation (1), we have
\[
(1 - \lambda)x_{u_i} = x_w = 0,
\]
so \( x_{u_i} = 0 \) since \( \lambda \neq 1 \). Similarly, for any vertex \( v \in V(T_{k,l}) \), we have \( x_v = 0 \).

Since \( x_{v_g} = x_{u'} = 0 \), from \( L(G)X = \lambda X \), we obtain that \( X|_{(C-v_g)} \in \mathbb{V}_{L_{v_g}(C)}^\lambda \) and \( X|_C \in \mathbb{V}_{L(C)}^\lambda \). Hence \( \lambda \) is an eigenvalue of both \( L_{v_g}(C) \) and \( L(C) \).

Note that \( L_{v_g}(C) = 2 - A(P_{g-1}) \), so
\[
\lambda = 2 - 2\cos\frac{k\pi}{g}, \qquad k = 1, 2, \dots, g-1.
\]
Thus \( \lambda \neq 0 \) or \( 4 \), and therefore \( \lambda \) is a double eigenvalue of \( L(C) \).

By the interlacing theorem, we have \( m_{L_{v_1}(G)}(\lambda) \ge 1 \).

Let \( Y \) be an eigenvector of \( L_{v_1}(G) \) corresponding to \( \lambda = 2 - 2\cos(k\pi/g) \), and let \( y_v \) denote the component of \( Y \) at vertex \( v \). Let \( y_{v_2} = a(\eta - \eta^{-1}) \), where \( \eta \) is a \( 2g \)-th root of unity. We know that \( a \neq 0 \); otherwise, by equation (1), we would have \( Y = 0 \), a contradiction.

Using equation (1), we obtain:
\[
\lambda y_{v_2} = 2y_{v_2} - y_{v_3} \quad\Longrightarrow\quad y_{v_3} = a(\eta^2 - \eta^{-2}),
\]
\[
\lambda y_{v_3} = 2y_{v_3} - y_{v_2} - y_{v_4} \quad\Longrightarrow\quad y_{v_4} = a(\eta^3 - \eta^{-3}),
\]
\[
\vdots
\]
\[
\lambda y_{v_{g-2}} = 2y_{v_{g-2}} - y_{v_{g-3}} - y_{v_{g-1}} \quad\Longrightarrow\quad y_{v_{g-1}} = a(\eta^{g-2} - \eta^{-(g-2)}),
\]
\[
\lambda y_{v_{g-1}} = 2y_{v_{g-1}} - y_{v_{g-2}} - y_{v_g} \quad\Longrightarrow\quad y_{v_g} = a(\eta^{g-1} - \eta^{-(g-1)}).
\]

Using \( y_{v_{g-1}} = a(\eta^{g-2} - \eta^{-(g-2)}) \) and \( y_{v_g} = a(\eta^{g-1} - \eta^{-(g-1)}) \), we obtain
\begin{equation}
\lambda y_{v_g} = 2y_{v_g} - y_{v_{g-1}}. \tag{4}
\end{equation}

If \( l > 1 \), i.e., \( u \neq w \), then from \( L_{v_1}(G)Y = \lambda Y \), we have
\[
\lambda y_{v_g} = 3y_{v_g} - y_{v_{g-1}} - y_{u'}.
\]
Combining this with (4), we obtain \( y_{v_g} = y_{u'} \neq 0 \). Since \( L_{v_1}(G)Y = \lambda Y \), we have
\[
\begin{aligned}
\lambda y_{u'} 
&= d_G(u')y_{u'} - y_{v_g} - \sum_{u' \sim v \in T(k,l) \setminus \{v_g\}} y_v \\
&= (d_G(u') - 1)y_{u'} - \sum_{u' \sim v \in T(k,l) \setminus \{v_g\}} y_v \\
&= d_{T(k,l-1)}(u')y_{u'} - \sum_{u' \sim v \in T(k,l-1)} y_v.
\end{aligned}
\]
This implies that \( Y|_{T_{k,l-1}} \in \mathbb{V}_{T_{k,l-1}}^\lambda \), and hence \( m_{T_{k,l-1}}(\lambda) = 1 \).

If \( l = 1 \), i.e., \( u = w \), then from \( L_{v_1}(G)Y = \lambda Y \), we have
\[
\lambda y_{v_g} = (k+2)y_{v_g} - y_{v_{g-1}} - y_{u_1} - y_{u_2} - \cdots - y_{u_k}.
\]
By symmetry of the graph, we easily obtain \( y_{u_1} = y_{u_2} = \cdots = y_{u_k} \). That is,
\[
\lambda y_{v_g} = (k+2)y_{v_g} - y_{v_{g-1}} - ky_{u_1}.
\]
Combining this with (4), we know that \( y_{u_1} = y_{v_g} \). Since
\[
\lambda y_{u_1} = y_{u_1} - y_{v_g} = 0,
\]
we have \( \lambda = 0 \) because \( y_{u_1} \neq 0 \), contradicting \( \lambda \neq 0 \).

We now prove the necessity. We have
\[
m_G(\lambda) \ge m_{G - e_{u,u'}}(\lambda) - 1 = m_C(\lambda) + m_{T_{k,l-1}}(\lambda) - 1 = 2.
\]
Since \( m_G(\lambda) \le 2 \), it follows that \( m_G(\lambda) = 2 \).
\end{proof}

\begin{figure}[H]
  \centering
  \includegraphics[width=1.0\linewidth]{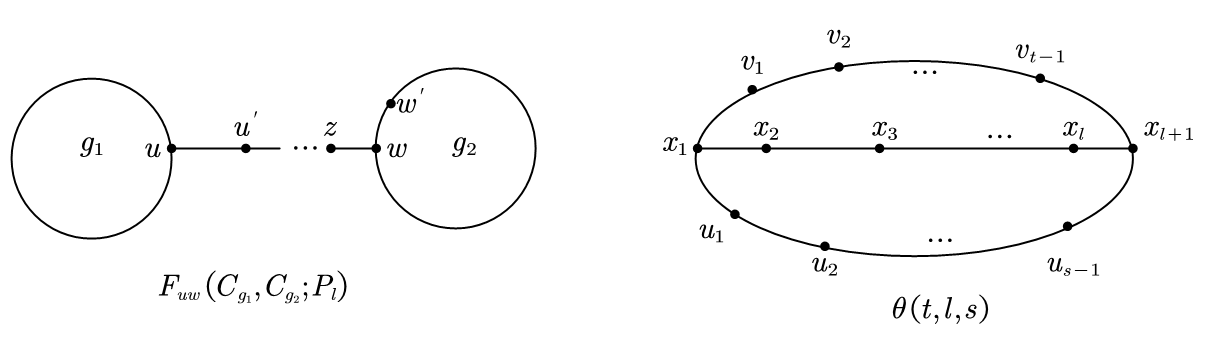}
\caption{$ G = F_{u,w}(C_{g_1}, C_{g_2}; P_l)$ and $G=\theta(t,l,s)$}\label{figure3}
\end{figure}

\begin{Lemma}
If \( c(G) = 2 \), \( q(G) = 0 \), and \( \lambda \neq 1 \), then \( G \) is \( \lambda \)-optimal if and only if 
\[
G = F_{uw}(C_{g_1}, C_{g_2}; P_l), \quad l > 1,
\]
\[
m_{C_{g_1}}(\lambda) = m_{C_{g_2}}(\lambda) = 2,
\]
and when \( l > 2 \),
\[
m_{P_{l-2}}(\lambda) = 1.
\]
\end{Lemma}

\begin{proof}
We first prove the sufficiency. We divide the proof into two cases.

\medskip
\noindent \textbf{Case 1.} \( G = F_{uw}(C_{g_1}, C_{g_2}; P_l) \), \( l \ge 1 \), as shown in Figure 3.

If \( l = 1 \), i.e., \( u = w \), let \( G' = G - e_{ww'} \). Then
\[
m_{G'}(\lambda) \ge m_G(\lambda) - 1 = 2.
\]
Since \( m_{G'}(\lambda) \le 2c(G') + q(G') - 1 = 2 \), we have \( m_{G'}(\lambda) = 2 \), meaning that \( G' \) is \( \lambda \)-optimal.

Thus, by Lemma 3.6, we obtain
\[
m_{C_{g_1}}(\lambda) = 2, \qquad m_{P_{g_1-1}}(\lambda) = 1.
\]
Similarly, we have
\[
m_{C_{g_2}}(\lambda) = 2, \qquad m_{P_{g_2-1}}(\lambda) = 1.
\]

Let \( \lambda = 2 - 2\cos(2k\pi/m) \), where \( (k, m) = 1 \), \( \lambda \neq 0, 4 \), and \( k \in [1, m] \). Then
\[
g_1 \equiv 0 \pmod m \quad \text{and} \quad 2(g_1 - 1) \equiv 0 \pmod m.
\]
That is,
\[
g_1 = am, \qquad 2g_1 = bm + 2 \quad \text{for some } a, b \in \mathbb{N}.
\]
Hence \( g_1 = (b - a)m + 2 \), i.e., \( g_1 \equiv 2 \pmod m \). Combining this with \( g_1 \equiv 0 \pmod m \), we obtain \( m = 2 \), so \( \lambda = 0 \) or \( 4 \), a contradiction.

If \( l \ge 2 \), we similarly have \( m_{C_{g_1}}(\lambda) = m_{C_{g_2}}(\lambda) = 2 \). Moreover, when \( l > 2 \), since \( G' = G - e_{ww'} \) is \( \lambda \)-optimal, we have
\[
m_{P_{g_2 + l - 2}}(\lambda) = 1.
\]
Thus,
\[
g_2 \equiv 0 \pmod m \quad \text{and} \quad 2(l + g_2 - 2) \equiv 0 \pmod m.
\]
A straightforward calculation yields \( 2(l - 2) \equiv 0 \pmod m \), which implies that \( P_{l-2} \) has \( \lambda \) as an eigenvalue.

\medskip
\noindent \textbf{Case 2.} \( G = \theta(t, l, s) \), as shown in Figure 3.

We know that at least two of \( t, l, s \) are greater than or equal to 2. Without loss of generality, assume \( t, l \ge 2 \).

Similar to the proof of Case 1, we know that both \( G - e_{v_{t-1}x_{l+1}} \) and \( G - e_{x_l x_{l+1}} \) are \( \lambda \)-optimal. Let
\[
\lambda = 2 - 2\cos(2k\pi/m),
\]
where \( (k, m) = 1 \), \( \lambda \neq 0, 4 \), and \( k \in [1, m] \). Then we have
\[
t + s \equiv 0 \pmod m, \qquad l + s \equiv 0 \pmod m,
\]
and similarly,
\[
t + l \equiv 0 \pmod m.
\]
This implies that \( 2t \equiv 0 \pmod m \).

Since \( G - e_{v_{t-1}x_{l+1}} \) is \( \lambda \)-optimal, we have
\[
2(t - 1) \equiv 0 \pmod m.
\]
Combining this with \( 2t \equiv 0 \pmod m \), we obtain \( m = 2 \), again a contradiction.

\medskip
\noindent We now prove the necessity.

If \( l = 2 \), then
\[
m_G(\lambda) \ge m_{G - e_{u,w}}(\lambda) - 1 = 3.
\]
Since \( m_G(\lambda) \le 2c(G) + q - 1 = 3 \), it follows that \( G \) is \( \lambda \)-optimal.

If \( l \ge 2 \), then
\[
3 \ge m_G(\lambda) \ge m_{G - e_{u,u'} - e_{zw}}(\lambda) - 2 = 2 + 2 + 1 - 2 = 3.
\]
Thus \( G \) is also \( \lambda \)-optimal.
\end{proof}

By Lemmas 3.6 and 3.7, we know that the conclusion of Theorem 2.1 holds.

 \noindent $\bullet$\ \ \ \ {\sf $s$-$p$-deletion}

Let $H$, $K_{1,k}$, and $P_l$ be three pairwise disjoint graphs, where $u \in V(H)$ with $d_H(u) \geq 2$, and $v$ is the center vertex of$K_{1,k}$.
Respectively identify $u$ and $v$ with two pendant vertices of a path $P_l$  (if $l = 1$, identifying $u$ and $v$), the resultant graph is denoted by $G$.
Conversely,  we say that $H$ is obtained from $G$ through an $s$-$p$-deletion operation.

\noindent $\bullet$\ \ \ \ {\sf $c$-$p$-deletion}

 Let $H$, $C_k$, and $P_l$ be three pairwise disjoint graphs, where $u \in V(H)$ and $d_H(u) \geq 2$, and $v$ is a vertex of $C_k$.
Respectively identify $u$ and $v$ with two pendant vertices of a path $P_l$ (if $l = 1$, identifying $u$ and $v$), we obtain  a graph $G$.
Conversely,  we say that $H$ is obtained from $G$ through an $c$-$p$-deletion operation.

\noindent $\bullet$\ \ \ \ {\sf $c$-deletion}

Let \( v_0 \sim v_1 \sim v_2 \sim \cdots \sim v_k \sim v_{k+1} \) be a sequence of vertices on some cycle of \( G \), where \( d_G(v_0) \ge 3 \), \( d_G(v_{k+1)} \ge 3 \), and \( d_G(v_1) = d_G(v_2) = \cdots = d_G(v_k) = 2 \). Then the graph
\[
G_1 = G - \{v_1, v_2, \dots, v_k\}
\]
is said to be obtained from \( G \) through a \( c\)-deletion operation.

\begin{Lemma}
Let \( G \) be a graph with \( c(G) + q(G) \ge 3 \), and let \( G_1 \) be a graph obtained from \( G \) by an \( s \)-\( p \)-deletion, a \( c \)-\( p \)-deletion, or a \( c \)-deletion operation. If \( \lambda \neq 1 \) and \( G \) is \( \lambda \)-optimal, then \( G_1 \) is also \( \lambda \)-optimal.
\end{Lemma}

\begin{proof}
We consider each operation separately.

\medskip
\noindent \textbf{\( s \)-\( p \)-deletion.} First, we note that \( c(G) = c(G_1) \) and \( q(G) = q(G_1) + 1 \). By Lemma 3.2(1), we have
\[
m_{G_1}(\lambda) \ge m_G(\lambda) - 1 = 2c(G) + q(G) - 1 - 1 = 2c(G_1) + q(G_1) - 1.
\]
On the other hand, we know that
\[
m_{G_1}(\lambda) \le 2c(G_1) + 2q(G_1) - 1,
\]
so \( m_{G_1}(\lambda) = 2c(G_1) + 2q(G_1) - 1 \), which means that \( G_1 \) is \( \lambda \)-optimal.

\medskip
\noindent \textbf{\( c \)-\( p \)-deletion.} This case is analogous to the \( s \)-\( p \)-deletion case and follows directly from Lemma 3.2(2).

\medskip
\noindent \textbf{\( c\)-deletion.} Let
\[
v_0 \sim v_1 \sim v_2 \sim \cdots \sim v_k \sim v_{k+1}
\]
be a sequence of vertices on some cycle of \( G \), where \( d_G(v_0) \ge 3 \), \( d_G(v_{k+1)} \ge 3 \), and \( d_G(v_1) = d_G(v_2) = \cdots = d_G(v_k) = 2 \). Let \( X \in \mathbb{V}_{L(G)}^\lambda \cap \mathbb{Z}_{L(G)}^{\{v_0, v_1\}} \), and let \( x_{v_i} \) denote the component of \( X \) at vertex \( v_i \). 
Then by equation (1), we obtain
\[
x_{v_i} = 0 \quad \text{for } i \in \{0, 1, \dots, k, k+1\}.
\]
Hence \( X|_{G_1} \in \mathbb{V}_{L(G_1)}^\lambda \).

Using the dimension formula, we easily obtain
\[
\dim(\mathbb{V}_{L(G)}^\lambda) \le \dim(\mathbb{V}_{L(G)}^\lambda \cap \mathbb{Z}_{L(G)}^{\{v_0, v_1\}}) + 2 \le \dim(\mathbb{V}_{L(G_1)}^\lambda) + 2.
\]
Thus
\[
m_{G_1}(\lambda) \ge m_G(\lambda) - 2 = 2c(G) + q(G) - 1 - 2.
\]
Noting that \( c(G) = c(G_1) + 1 \) and \( q(G) = q(G_1) \), we have
\[
m_{G_1}(\lambda) \ge 2c(G_1) + q(G_1) - 1.
\]
Since \( m_{G_1}(\lambda) \le 2c(G_1) + q(G_1) - 1 \), it follows that
\[
m_{G_1}(\lambda) = 2c(G_1) + q(G_1) - 1,
\]
i.e., \( G_1 \) is \( \lambda \)-optimal.
\end{proof}

\begin{figure}[H]
  \centering
  \includegraphics[width=1.0\linewidth]{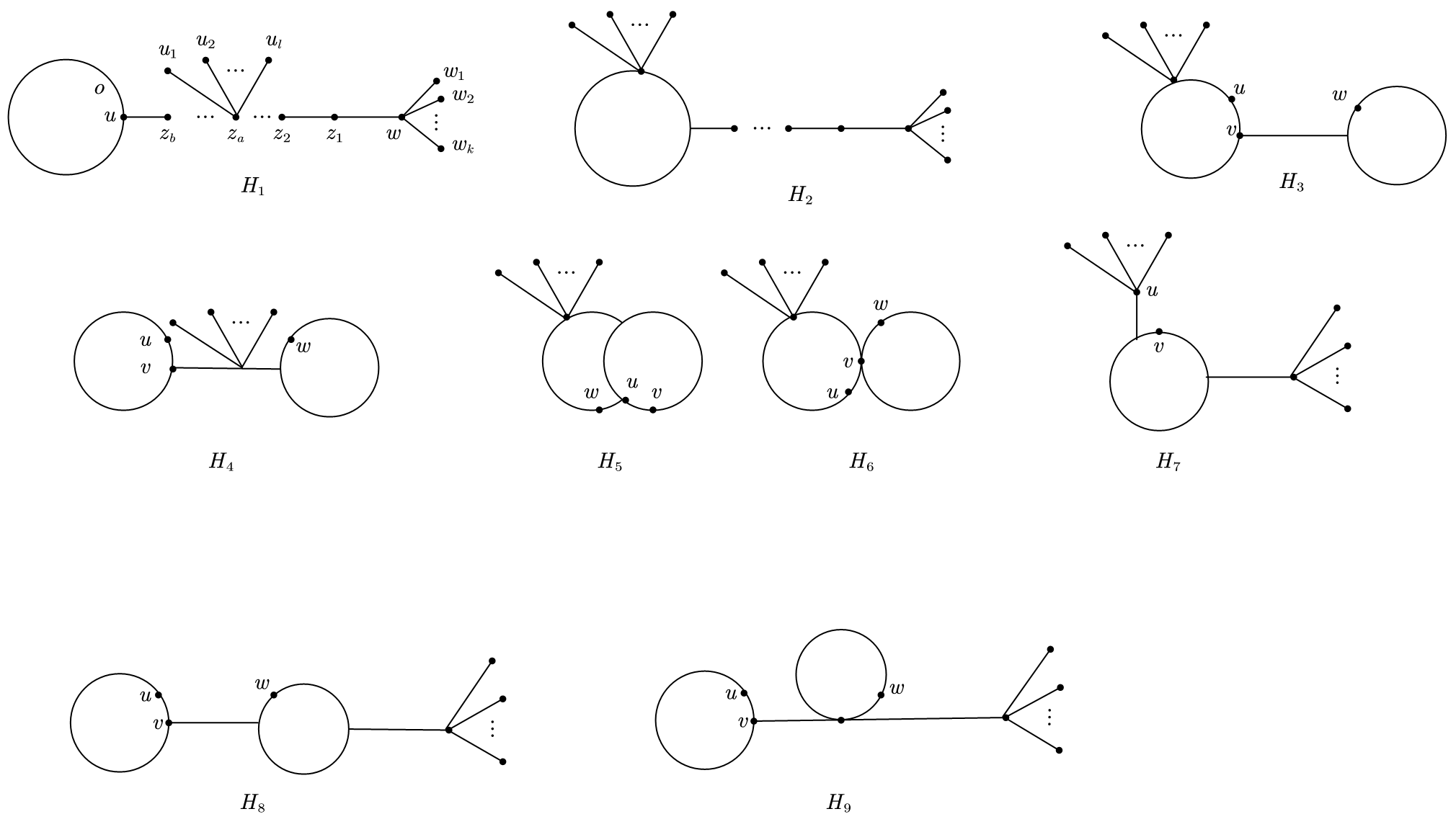}
\caption{$ c(G)+q(G)=3$}\label{figure4}
\end{figure}

\begin{Lemma}
Let \( \lambda \neq 1 \). If \( c(G) + q(G) \ge 3 \) and \( G \) is \( \lambda \)-optimal, then the following hold:

(1) If \( q(G) > 1 \), let \( G = F_{uv}(H, K_{1,k_i}; P_l) \), where \( d_G(u) \ge 3 \). Then \( l \ge 2 \); in other words, every \( K_{1,k_i} \) is a pendant structure.

(2) If \( c(G) > 1 \), then every cycle of \( G \) is a pendant cycle.
\end{Lemma}

\begin{proof}
We proceed by induction on \( c(G) + q(G) \).

\medskip
\noindent \textbf{Case 1:} \( c(G) + q(G) = 3 \).

We first prove that every \( K_{1,k_i} \) is a pendant-\( K_{1,k_i} \). We divide the proof into the following subcases.

\medskip
\noindent \textbf{Case 1.1:} \( c(G) = 0 \). Then \( G \) is a tree, and the conclusion is obvious from the result of \cite{Li2026}.

\medskip
\noindent \textbf{Case 1.2:} \( c(G) = 1 \). Suppose that there exists a \( K_{1,k_1} \) which is not a pendant structure. Then \( G \) is isomorphic to \( H_1 \) or \( H_2 \) as shown in Figure 4. We show that in this case \( G \) is not \( \lambda \)-optimal.

Assume first that \( G \) is isomorphic to \( H_1 \) in Figure 3. Let \( X \in \mathbb{V}_{L(G)}^\lambda \cap \mathbb{Z}_{L(G)}^{\{o,w\}} \), and let \( x_{v_i} \) denote the component of \( X \) at vertex \( v_i \). By equation (1), we have
\[
(\lambda - 1)x_{w_j} = x_w = 0,
\]
so \( x_{w_j} = 0 \) for \( 1 \le j \le k \). Similarly,
\[
(d_w - \lambda)x_w = \sum_{1 \le j \le k} x_{w_j} + x_{z_1},
\]
hence \( x_{z_1} = 0 \). Continuing this argument, we obtain \( x_v = 0 \) for every vertex \( v \in V(G) \), i.e., \( X = 0 \). Therefore, by the dimension formula, we easily get
\[
\dim(\mathbb{V}_{L(G)}^\lambda) \le \dim(\mathbb{V}_{L(G)}^\lambda \cap \mathbb{Z}_{L(G)}^{\{o,w\}}) + 2 = 2 < 2c(G) + q(G) - 1 = 3,
\]
so \( G \) is not \( \lambda \)-optimal.

If \( G \) is isomorphic to \( H_2 \) in Figure 4, a similar argument shows that \( G \) is not \( \lambda \)-optimal.

\medskip
\noindent \textbf{Case 1.3:} \( c(G) = 2 \). Suppose that there exists a \( K_{1,k_1} \) which is not a pendant structure. Then \( G \) is isomorphic to one of \( H_3, H_4, H_5, H_6 \) in Figure 3. Let
\[
X \in \mathbb{V}_{L(G)}^\lambda \cap \mathbb{Z}_{L(G)}^{\{u,v,w\}}.
\]
Similar to the discussion in Case 1.2, we easily obtain \( X = 0 \). Combining this with the dimension formula, we have
\[
\dim(\mathbb{V}_{L(G)}^\lambda) \le \dim(\mathbb{V}_{L(G)}^\lambda \cap \mathbb{Z}_{L(G)}^{\{u,v,w\}}) + 3 = 3 < 2c(G) + q(G) - 1 = 4,
\]
so \( G \) is not \( \lambda \)-optimal.

\medskip
\noindent We now prove that every cycle is a pendant cycle.

\medskip
\noindent \textbf{Case 1.4:} \( c(G) = 1 \). If the unique cycle of \( G \) is not a pendant cycle, then \( G \cong H_7 \). Let \( X \in \mathbb{V}_{L(G)}^\lambda \cap \mathbb{Z}_{L(G)}^{\{u,v\}} \). Similarly, we obtain \( X = 0 \), i.e.,
\[
m_G(\lambda) \le 2 < 2c(G) + q(G) - 1 = 3,
\]
so \( G \) is not \( \lambda \)-optimal.

\medskip
\noindent \textbf{Case 1.5:} \( c(G) = 2 \). Let \( C_{g_1} \) and \( C_{g_2} \) be two distinct cycles of \( G \). Suppose that \( C_{g_1} \) is a non-pendant cycle. 
If \( C_{g_1} \) and \( C_{g_2} \) have at least one common vertex, then by applying an \( s \)-\( p \)-deletion to \( G \), we obtain \( G_1 \) such that
\[
G_1 \cong F_{uv}(C_{g_1}, C_{g_2}; P_1) \quad \text{or} \quad G_1 \cong \theta(t,s,l).
\]
By Lemma 3.7, we know that \( G_1 \) is not \( \lambda \)-optimal, contradicting Lemma 3.8.

If \( C_{g_1} \) and \( C_{g_2} \) have no common vertex, then \( G \cong H_8 \) or \( H_9 \). Let
\[
X \in \mathbb{V}_{L(G)}^\lambda \cap \mathbb{Z}_{L(G)}^{\{u,v,w\}}.
\]
Then we have \( X = 0 \), i.e.,
\[
m_G(\lambda) \le 3 < 2c(G) + q(G) - 1 = 4,
\]
so \( G \) is not \( \lambda \)-optimal, a contradiction.

\medskip
\noindent \textbf{Case 1.6:} \( c(G) = 3 \). Suppose that \( G \) has a non-pendant cycle. Then by applying a \( c \)-\( p \)-deletion or a \( c \)-deletion to \( G \), we obtain \( G_1 \) which still contains a non-pendant cycle, and
\[
G_1 \cong F_{uv}(C_{g_1}, C_{g_2}; P_1) \quad \text{or} \quad G_1 \cong \theta(t,s,l).
\]
By Lemma 3.7, we know that \( G_1 \) is not \( \lambda \)-optimal, contradicting Lemma 3.8.

\medskip
\noindent Thus the conclusion holds when \( c(G) + q(G) = 3 \).

\noindent \textbf{Case 2:} \( c(G) + q(G) \geq 3 \).

\medskip
\noindent Assume that the conclusion holds for \( c(G) + q(G) = k \ge 3 \). We now prove that it holds for \( c(G) + q(G) = k + 1 \).

\medskip
\noindent \textbf{Step 1:} Apply a \( c \)-\( p \)-deletion, a \( c \)-deletion, or an \( s \)-\( p \)-deletion to \( G \) to obtain \( G_1 \) such that \( c(G_1) + q(G_1) = k \). By Lemma 3.8, \( G_1 \) is \( \lambda \)-optimal. Hence all cycles of \( G_1 \) are pendant cycles,  and all \( K_{1,k_i} \) in \( G_1 \) are pendant-\( K_{1,k_i} \).

\medskip
\noindent \textbf{Step 2:} Apply a \( c \)-\( p \)-deletion, a \( c \)-deletion, or an \( s \)-\( p \)-deletion to \( G \) to obtain \( G_2 \neq G_1 \) such that \( c(G_2) + q(G_2) = k \). By Lemma 3.8, \( G_2 \) is \( \lambda \)-optimal. Hence all cycles of \( G_2 \) are pendant cycles, and all \( K_{1,k_j} \) in \( G_2 \) are pendant-\( K_{1,k_j} \).

Combining Steps 1 and 2, we conclude that all cycles of \( G \) are pendant cycles, and all \( K_{1,k_i} \) are pendant-\( K_{1,k_i} \).
\end{proof}

\section{Proof of results}

\begin{figure}[H]
  \centering
  \includegraphics[width=1.0\linewidth]{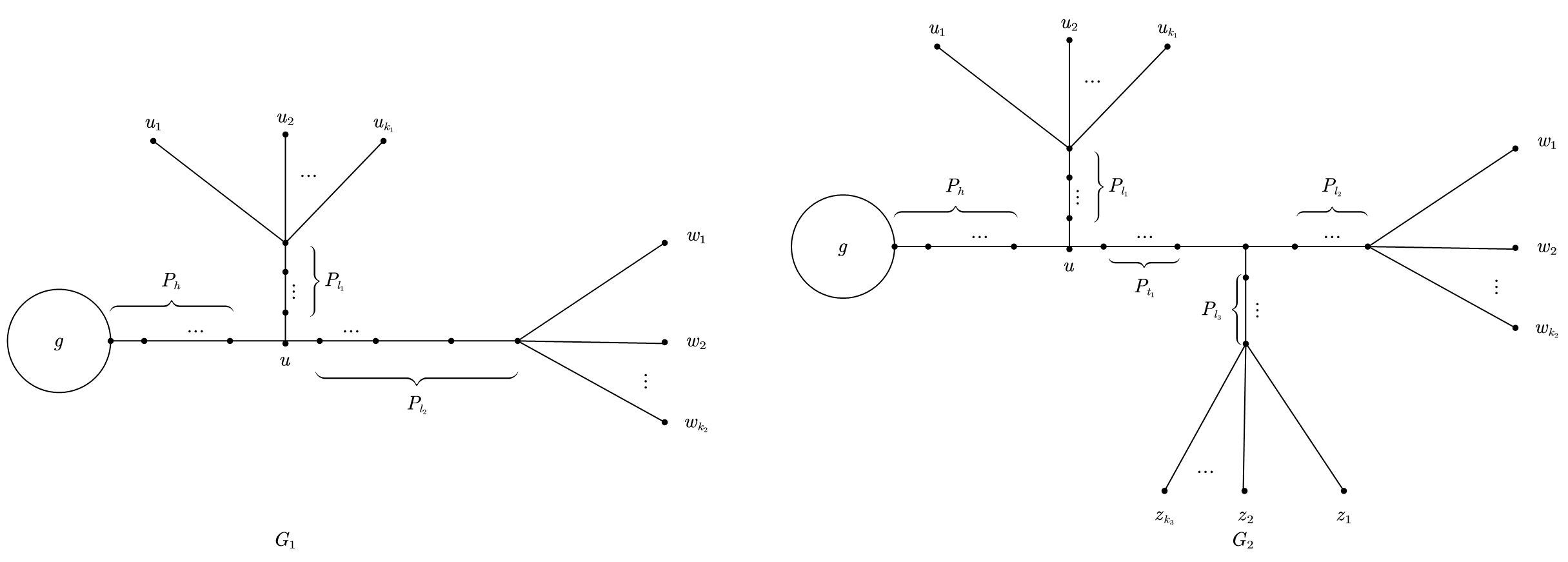}
\caption{$ G_1$ and $G_2$}\label{figure5}
\end{figure}
To prove Theorem 2.2, we first consider the case \( c(G) = 1 \), which is established in Lemma 4.1. We then proceed by induction on \( c(G) \) to complete the proof of Theorem 2.2.

\begin{Lemma}
Let \( G \) be a unicyclic graph, \( \lambda \neq 1 \), and \( c(G) + q(G) \ge 3 \). Suppose
\[
G - \Gamma(G) = C(g,h) \cup \mathcal{T}_G \cup \mathcal{P}_G,
\]
where \( \mathcal{T}_G \) is the union of \( T(k_i, l_i) \) for \( 1 \le i \le q \), and \( \mathcal{P}_G \) is the union of \( P_{t_j} \) for \( 1 \le j \le |\Gamma(G)| - 1 \). 
Then \( G \) is \( \lambda \)-optimal if and only if
\[
\begin{aligned}
&m_{Q_{g,h}}(\lambda) = 2 , \\
&m_{B_{k_i,l_i}}(\lambda) = 1 &&(1 \le i \le q(G)), \\
&m_{U_{t_j}}(\lambda) = 1 &&(1 \le j \le |\Gamma(G)| - 1).
\end{aligned}
\]
\end{Lemma}

\begin{proof}
We first prove the sufficiency. We proceed by induction on \( q(G) \).

When \( q(G) = 2 \), we have \( G \cong G_1 \) as shown in Figure 5. If
\[
m_{B_{k_1,l_1}}(\lambda) = m_{B_{k_2,l_2}}(\lambda) = 0,
\]
then by the interlacing theorem,
\[
m_G(\lambda) \le m_{L_u(G)}(\lambda) + 1
= m_{Q_{g,l}}(\lambda) + m_{B_{k_1,l_1}}(\lambda) + m_{B_{k_2,l_2}}(\lambda)
= m_{Q_{g,l}}(\lambda) \le 2
< 2c(G) + q(G) - 1 = 3,
\]
a contradiction. Thus, without loss of generality, assume \( m_{B_{k_1,l_1}}(\lambda) = 1 \). 
Lemma 3.3 then gives \( m_{B_{k_3,l_3-1}}(\lambda) = 0 \). Consequently, by Lemma 3.5,
\[
m_G(\lambda) = m_{B_{k_2,l_2}}(\lambda) + m_{Q_{g,l}}(\lambda) = 3.
\]
Since \( m_{B_{k_2,l_2}}(\lambda) \le 1 \) and \( m_{Q_{g,l}}(\lambda) \le 2 \), we obtain
\[
m_{B_{k_2,l_2}}(\lambda) = 1, \qquad m_{Q_{g,l}}(\lambda) = 2.
\]
Thus the conclusion holds.

Assume that the conclusion holds for \( q(G) = k \ge 2 \). We now prove it for \( q(G) = k + 1 \).

\medskip
\noindent \textbf{Case 1.} \( G \cong G_2 \), where \( G_2 \) is as shown in Figure 5.

\medskip
\noindent \textbf{Step 1:} Apply an \( s \)-\( p \)-deletion to \( G \) by removing \( T(k_2, l_2) \), obtaining \( G'_2 \). By Lemma 3.8, \( G'_2 \) is \( \lambda \)-optimal and \( q(G'_2) = k \). Hence
\[
m_{Q_{g,l}}(\lambda)=2, \qquad m_{B_{k_1,l_1}}(\lambda) = 1, \qquad m_{B_{k_3,l_3+t_1+1}}(\lambda) = 1.
\]

\medskip
\noindent \textbf{Step 2:} Apply an \( s \)-\( p \)-deletion to \( G \) by removing \( T(k_1, l_1) \), obtaining \( G''_2 \). Again, \( G''_2 \) is \( \lambda \)-optimal and \( q(G''_2) = k \). Hence
\[
m_{B_{k_3,l_3}}(\lambda) = 1, \qquad m_{B_{k_2,l_2}}(\lambda) = 1.
\]
Since \( m_{B_{k_3,l_3}}(\lambda) = 1 \), by Lemma 3.3 we have \( m_{B_{k_3,l_3-1}}(\lambda) = 0 \). Then by Lemma 3.5,
\[
m_{B_{k_3,l_3+t_1+1}}(\lambda) = m_{U_{t_1}}(\lambda) = 1,
\]
as desired.

\medskip
\noindent \textbf{Case 2.} \( G \not\cong G_2 \). Similar to Case 1, we proceed in two steps.

\medskip
\noindent \textbf{Step 1:} Apply an \( s \)-\( p \)-deletion to \( G \) by removing \( T(k_1, l_1) \), obtaining \( G' \) such that \( q(G') = k \). By Lemma 3.7, \( G' \) is \( \lambda \)-optimal, so the conclusion holds for \( G' \).

\medskip
\noindent \textbf{Step 2:} Apply an \( s \)-\( p \)-deletion to \( G \) by removing \( T(k_2, l_2) \), obtaining \( G'' \) such that \( q(G'') = k \). The conclusion also holds for \( G'' \).

It is worth noting that since \( G \not\cong G_2 \), by Steps 1 and 2, we can always ensure that \( C(g,h) \) is either a pendant-\( C(g,h) \) in \( G' \) or a pendant-\( C(g,h) \) in \( G'' \).
Similarly, for every pendant-\( T(k_i, l_i) \) in \( G \), it is either a pendant-\( T(k_i, l_i) \) in \( G' \) or a pendant-\( T(k_i, l_i) \) in \( G'' \). The same conclusion holds for the internal paths in \( \mathcal{P}_G \). Therefore the desired conclusion follows.

We now prove the necessity. By the interlacing theorem, we have
\[
\begin{aligned}
m_G(\lambda) 
&\ge m_{L_{\Gamma(G)}}(\lambda) - |\Gamma(G)| \\
&= m_{Q(g,l)} + \sum_{1 \le i \le q(G)} m_{B(k_i,l_i)}(\lambda) 
   + \sum_{1 \le j \le |\Gamma(G)|-1} m_{U_{t_j}}(\lambda) - |\Gamma(G)| \\
&= 2 + q(G) + |\Gamma(G)| - 1 - |\Gamma(G)| \\
&= 2 + q(G) - 1.
\end{aligned}
\]
On the other hand,
\[
m_G(\lambda) \le 2c(G) + q(G) - 1 = 2 + q(G) - 1.
\]
Therefore \( m_G(\lambda) = 2c(G) + q(G) - 1 \), which means that \( G \) is \( \lambda \)-optimal.

\end{proof}

\begin{figure}[H]
  \centering
  \includegraphics[width=1.0\linewidth]{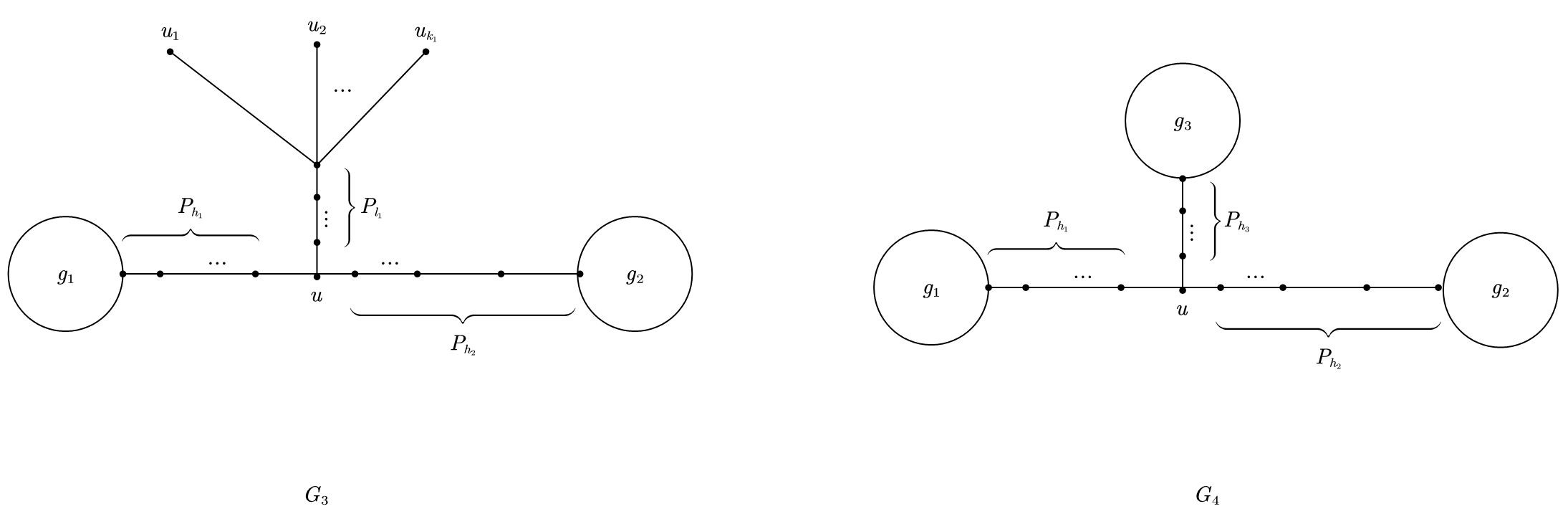}
\caption{$ G_3$ and $G_4$}\label{figure6}
\end{figure}

When \( c(G) + q(G) \ge 3 \), by Lemma 4.1, we know that if a unicyclic graph \( G \) is \( \lambda \)-optimal, then \( m_{Q_{g,h}}(\lambda) = 2 \), where \( Q_{g,h} \) is defined as in Lemma 4.1. We now aim to determine the necessary and sufficient conditions for \( m_{Q_{g,h}}(\lambda) = 2 \). The following lemma provides the answer.

\begin{Lemma}

Let \( G = C(g,l+1) \) and \( \lambda \neq 1 \). Then \( m_{Q(g,l)}(\lambda) = 2 \) if and only if \( m_{C_g}(\lambda) = 2 \), \( l > 2 \), and \( m_{L_v(P_l)}(\lambda) = 1 \), where \( v \) is an arbitrary pendant vertex of \( P_l \).
\end{Lemma}
\begin{proof}
The proof is almost identical to that of Lemma 3.6, and interested readers may complete the proof themselves.
\end{proof}
\begin{proof}[Proof of Theorem 2.2]
We first prove the sufficiency. We begin by showing that the conclusion holds when \( c(G) + q(G) = 3 \).

\medskip
\noindent \textbf{Case 1.} \( c(G) = 1, q(G) = 2 \). The conclusion already holds.

\medskip
\noindent \textbf{Case 2.} \( c(G) = 2, q(G) = 1 \). Then \( G \cong G_4 \) as shown in Figure 6. Suppose that
\[
m_{Q_{g_1,h_1}}(\lambda) \le 1, \qquad m_{Q_{g_2,h_2}}(\lambda) \le 1.
\]
By the interlacing theorem, we have
\[
m_G(\lambda) \le m_{L_u(G)}(\lambda) - 1
= m_{Q_{g_1,h_1}}(\lambda) + m_{Q_{g_2,h_2}}(\lambda) + m_{B_{k_1,l_1}}(\lambda)
\le 3
< 2c(G) + q(G) - 1 = 4,
\]
a contradiction. Without loss of generality, assume \( m_{Q_{g_1,h_1}}(\lambda) = 2 \). Then by Lemma 3.4, we have \( m_{Q_{g_1,h_1-1}}(\lambda) = 1 \). Therefore, by Lemma 3.5, we obtain
\[
m_G(\lambda) = m_{Q_{g_1,h_1-1}}(\lambda) + m_{Q_{g_2,h_2}}(\lambda) + m_{B_{k_1,l_1}}(\lambda) = 4.
\]
Since \( m_{Q_{g_2,h_2}}(\lambda) \le 2 \) and \( m_{B_{k_1,l_1}}(\lambda) \le 1 \), we must have
\[
m_{Q_{g_2,h_2}}(\lambda) = 2, \qquad m_{B_{k_1,l_1}}(\lambda) = 1,
\]
and the conclusion follows.

\medskip
\noindent \textbf{Case 3.} \( c(G) = 3, q(G) = 0 \). Then \( G \cong G_5 \) as shown in Figure 6. The proof is analogous to that of Case 2.

\medskip
\noindent For the case \( c(G) + q(G) \ge 3 \), we proceed by induction on \( c(G) \). Using Lemma 3.8, the proof is entirely similar to that of Theorem 4.1, and is therefore omitted.

\medskip
\noindent \textbf{Necessity.} By the interlacing theorem, we have
\begin{align*}
m_G(\lambda) 
&\ge m_{L_{\Gamma(G)}}(\lambda) - |\Gamma(G)| \\
&= \sum_{1 \le m \le c(G)} m_{Q(g_m,h_m)}(\lambda) 
   + \sum_{1 \le i \le q(G)} m_{B(k_i,l_i)}(\lambda) 
   + \sum_{1 \le j \le |\Gamma(G)|-1} m_{U_{t_j}}(\lambda) - |\Gamma(G)| \\
&= 2c(G) + q(G) + |\Gamma(G)| - 1 - |\Gamma(G)| \\
&= 2c(G) + q(G) - 1.
\end{align*}
On the other hand, \( m_G(\lambda) \le 2c(G) + q(G) - 1 \). Therefore \( m_G(\lambda) = 2c(G) + q(G) - 1 \), which means that \( G \) is \( \lambda \)-optimal, and the conclusion follows.    
\end{proof}
\begin{proof}[Proof of Theorem 2.3]

The conclusion follows immediately from Lemma 4.2 and Theorem 2.2.
\end{proof}

\section{Acknowledgments}

We gratefully acknowledge the support of the Graduate Innovation Program of China university of Mining and Technology (2025WLKXJ146) for this work.

\end{document}